\documentclass[a4paper,12pt]{scrartcl}

\usepackage[bookmarks,pagebackref]{hyperref}
\usepackage{amsrefs}
\usepackage[utf8]{inputenc}
\usepackage[english]{babel}
\usepackage{amssymb}
\usepackage{amsmath}
\usepackage{latexsym}
\usepackage{amsthm}
\usepackage{eucal}
\usepackage{bbm}
\usepackage{chngcntr}
\usepackage{apptools}
\usepackage{mathtools}
\usepackage{authblk}
\usepackage[pdflatex]{crop}
\usepackage{tikz}
\usepackage{tikz-cd}
\usepackage{microtype}
\usepackage{enumitem}
\usepackage{graphicx}
\usepackage{mathrsfs}
\usepackage[colorinlistoftodos]{todonotes}
\usepackage{yfonts}
\usepackage{tabularx}
\usepackage{color}

\allowdisplaybreaks

\setkomafont{disposition}{\normalfont\bfseries}
\newcommand{\auth}[0]{{Renan Assimos and J\"urgen Jost}}
\newcommand{\tit}[0]{{Spherical Bernstein theorems for codimension 1 and 2}}
\newcommand{\kw}[0]{{Bernstein theorem, minimal graph, harmonic map, Grassmannian, Gauss map, maximum principle}}

\hypersetup{
pdfauthor={\auth},%
pdftitle={\tit},%
colorlinks, linktocpage=true, pdfstartpage=1, pdfstartview=FitV,%
breaklinks=true, pdfpagemode=UseNone, pageanchor=true, pdfpagemode=UseOutlines,%
plainpages=false, bookmarksnumbered, bookmarksopen=true, bookmarksopenlevel=1,%
hypertexnames=true, pdfhighlight=/O,%
urlcolor=black, linkcolor=black, citecolor=black, 
}
\numberwithin{equation}{section}

\theoremstyle{plain}
\newtheorem{thm}{Theorem}[subsection]
\AtAppendix{\counterwithin{thm}{section}}
\newtheorem{defn}[thm]{Definition}

\newtheorem{prop}[thm]{Proposition}
\newtheorem{cor}[thm]{Corollary}

\newtheorem{rmk}[thm]{Remark}
\newtheorem{remark}[thm]{Remark}

\theoremstyle{definition}

\newtheorem{eg}[thm]{Example}

\newcommand{\C}{\mathbb{C}}
\newcommand{\R}{\mathbb{R}}

\DeclareMathOperator{\tr}{tr}

\newcommand{\bigslant}[2]{{\raisebox{.2em}{$#1$}\left/\raisebox{-.2em}{$#2$}\right.}}

\let\originalleft\left
\let\originalright\right
\renewcommand{\left}{\mathopen{}\mathclose\bgroup\originalleft}
\renewcommand{\right}{\aftergroup\egroup\originalright}

\title{\tit}
\author{\auth\thanks{Correspondence: \href{mailto:assimos@mis.mpg.de}{assimos@mis.mpg.de}, \href{mailto:jjost@mis.mpg.de}{jjost@mis.mpg.de}}}
\affil{\small Max Planck Institute for Mathematics in the Sciences\\ Leipzig, Germany}
\date{}

\begin{document}

\maketitle

\begin{abstract}
A  result of B.Solomon  \cite{solomon84} says that  a compact minimal hypersurface $M^k$ of the sphere $S^{k+1}$ with $H^1(M)=0$,  whose Gauss map omits a neighborhood of an $S^{k-1}$ equator, is totally geodesic in $S^{k+1}$. We develop a new proof strategy which can also obtain an analogous result for codimension 2 compact minimal submanifolds of $S^{k+1}$.
\end{abstract}


\tableofcontents

\section{Introduction}

A fundamental result in the theory of minimal surfaces is Bernstein's theorem, which states that the only entire minimal graphs in Euclidean 3-space are planes. In other words, if $f(x,y)$ is a smooth function defined on all of $\R^2$ whose graph $(x,y,f(x,y))$ in $\R^3$ is a minimal surface, then $f$ is a linear function, and its graph a plane.

When the general theory of nonlinear PDEs was developed, it was a natural question whether Bernstein's theorem also holds in higher dimensions. In fact, this became one of the guiding problems of geometric analysis. On the positive side, the development lead to 
 the theorem of J. Simons \cite{simons68},  that in dimensions at most 7, any entire minimal graph is planar, but on the negative side,  E. Bombieri, E. de Giorgi, and E. Giusti \cite{bombieri69} constructed  a non-flat minimal graph in  dimension 8, and this then of course extends to higher dimensions. When one assumes, however, that  the slope of the graph is uniformly bounded,  a Bernstein-type result holds in arbitrary dimension by a theorem of J. Moser \cite{moser61}. This is the situation for codimension 1. For codimension, in \cite{assimos2018}, we could prove that Moser's Bernstein theorem is still. On the other hand, B. Lawson and R. Osserman \cite{lawson77} gave explicit counterexamples to Bernstein-type results in codimension greater or equal to 3. In fact, the cone over a Hopf map is an entire Lipschitz solution to the minimal surface system of codimension 3. Since the slope of the graph of such a cone is bounded, even a Moser-type result for codimension greater than 2 cannot hold. 

S.S. Chern \cite{chern70} has introduced the Spherical Bernstein problem. That problem concerns compact $(n-1)$-dimensional minimal submanifolds of the sphere $S^{n+m-1}$, and analogously to the Euclidean Bernstein problem just described, the aim is to prove that they are totally geodesic (i.e. equatorial) subspheres when their normal planes do not change their directions too much. Actually, the fact that this problem is concerned with {\em compact} minimal submanifolds avoids some of the technical problems of the Euclidean Bernstein problem. 

Following the strategy Hildebrandt-Jost-Widman \cite{hildebrandt80}, we shall prove Bernstein theorems by showing that the Gauss map of a minimal submanifold has to be constant under appropriate conditions. More precisely, these conditions stipulate that the image of the Gauss map omits a certain subset of its target, a Grassmann manifold. For instance, the Gauss map of a minimal Euclidean graph of codimension 1 is contained in an open hemisphere (the Grassmann manifold reduces to a sphere in the codimension 1 case). More generally, it suffices that the Gauss map omits half an equator \cite{jost2012}. In the spherical case, Solomon showed that it even suffices that the Gauss map omits a codimension two sphere, provided $M$ is a minimal hypersurface in the sphere with $H_1(M)=0$. In this paper, we develop a geometric approach to the Gauss map, drawing upon subtler aspects of the geometry of Grassmannians, that recovers Solomon's result and derives an analogous result in codimension 2, see Thm. \ref{mainthm} below. Since the Gauss map of a minimal submanifold of a Euclidean space or a sphere is harmonic by the Theorem of Ruh-Vilms, the key of our proof is a geometric maximum principle for harmonic maps, and in this context, we prove Theorem \ref{Maximal sets in $S^k$} which may be of wider interest. 

\color{black} While Spherical Bernstein theorems are interesting in their own right, they also find  strong applications in geometric analysis. Let us now recall the relation between the Spherical and the Euclidean Bernstein problems, and then state our results for the latter. Fleming’s \cite{fleming62} idea was that by re-scaling a nontrivial minimal graph in Euclidean space, one obtains a non flat minimal cone, and the intersection of that cone with the unit sphere is then a compact minimal submanifold of the latter. Therefore, conditions ruling out the latter can be translated into conditions ruling out the former; that is, spherical Bernstein theorems can prove Euclidean Bernstein theorems.  Because of the non-compact nature of minimal graphs in Euclidean space, as an important technical ingredient, one needs to invoke Allard’s regularity theory \cite{allard72}.

We thank Mario Micallef, Fernando Coda Marques, and Slava Matveev for suggestions and  helpful discussions. As well as Florio Ciaglia and Zachary Adams for helpful comments.

\section{Preliminaries}
\subsection{Harmonic maps}

Throughout this paper, $(M,g)$ and $(N,h)$ will denote Riemannian manifolds without boundaries. We will see $N \hookrightarrow (\R^L,g_{\text{euc}})$ isometrically embedded in Euclidean space (without further mention). 

Let us start recalling that the Sobolev space $W^{1,2}(M,N)$ is defined as follows.
\begin{align*}
W^{1,2}(M,N) \overset{\text{def}}{=} \,\,\,\, &\Big\{v: M\longrightarrow \R^L; \hspace{0.2cm}\left|\left|v\right|\right|^{2}_{\,W^{1,2}(M)}=\int_{M}(\Vert dv\Vert^{2} + \left|v\right|^2)\,\,dv_g< +\infty \,\,\text{and}\\
&v(x)\in N \,\,\text{for a.e.}\,\, x\in M \Big\}.
\end{align*}

\begin{defn}
    A harmonic map $\phi \in W^{1,2}(M,N)$ is a critical point of the energy functional, defined as follows.
        \begin{equation*}
        E(\phi) \overset{\text{def}}{=} \frac{1}{2} \int_{M} \Vert d\phi\Vert^{2}dvol_{g},
        \end{equation*}
        where $\Vert \cdot\Vert^{2} = \langle \cdot ,\cdot\rangle$ is the metric induced by $g$ and $h$ over the bundle $T^{*}M\otimes \phi^{-1}TN$. In local coordinates, one writes
        \begin{equation*}
        \Vert d\phi\Vert^2 = g^{ij}\frac{\partial\phi^{\beta}}{\partial x^i}\frac{\partial\phi^{\gamma}}{\partial x^j}h_{\beta\gamma}.
        \end{equation*}
\end{defn}

\noindent Deriving the Euler-Lagrange equations for $E(\cdot)$ and writing in coordinates, we get the following equation.
\begin{equation*}
0 = \tau(\phi) = \big(\Delta_{g}\phi^{\alpha} + g^{ij}\Gamma_{\beta\gamma}^{\alpha}\frac{\partial\phi^{\beta}}{\partial x^i}\frac{\partial\phi^{\gamma}}{\partial x^j}h_{\beta\gamma}\big)\frac{\partial}{\partial\phi^{\alpha}},
\end{equation*}
where  $\Gamma_{\beta\gamma}^{\alpha}$ denote the Christoffel symbols of $N$ and $\tau$ is called the tension field.

\begin{defn}[Gauss map]\label{gauss map}
        Let $M^p$ be a p-dimensional oriented embedded submanifold in the Euclidean space $\R^n$. For a given $x\in M$, the tangent space $T_xM$ is moved to the origin by parallel translation. This yields an oriented $p$-subspace of $\R^n$, i.e., a point in the oriented Grassmannian manifold $G^{+}_{p,n}$. This defines the Gauss map of the embedding $M^p\hookrightarrow \R^n$, denoted by $\gamma: M\longrightarrow G^+_{p,n}$.
\end{defn}

\begin{thm}[Ruh-Vilms \cite{ruh70}]
        Let $M^p\hookrightarrow \R^n$ be like in Definition~\ref{gauss map}. Then $\gamma$ is a harmonic map if and only if $M$ has parallel mean curvature.
\end{thm}

Using Ruh-Vilms' theorem, one may attempt to find subsets $V\subset G^{+}_{p,n}$ admitting no existence of non-constant harmonic maps $\phi:(M,g)\longrightarrow (G^+_{p,n},h_{\text{hom}})$, defined on a compact manifold $(M,g)$, with $\phi(M)\subset V$. Once such subset is found, one imposes conditions on the the Gauss image of a minimal submanifold of Euclidean space to obtain Bernstein-type results (roughly speaking, there are many issues with non-compactness of the domain that need to be carried out, among other things). In this regard, it is useful to remember  the following composition formula.
\begin{equation}
\tau(\psi\circ\phi) = d\psi\circ\tau(\phi) + \tr\nabla d\psi(d\phi,d\phi).
\end{equation}
When $\phi:(M,g)\longrightarrow (N,h)$, and $\psi: (N,h)\longrightarrow (P,i)$ is another map taking values on a Riemannian manifold $(P,i)$.

If $\phi$ is harmonic, the formula is reduced to $\tau(\psi\circ\phi)=\tr \nabla d\psi(d\phi,d\phi)$. Moreover, if $(P,i)=(\R,g_{\text{euc}})$, and $\psi$ is strictly convex, we can use the Maximum principle to prove the following.

\begin{prop}
        Let $M$ be a compact manifold (remember we always assume $\partial M=\emptyset$). Let $\phi:(M,g)\longrightarrow (N,h)$ be a harmonic map with $\phi(M)\subset A \subset N$. In addition, assume there exists a strictly convex function $f: A\longrightarrow \R$, in the geodesic sense. Then $\phi$ is constant.
\end{prop}

\noindent For compact minimal submanifolds of spheres $M^k\subset S^{n}\subset \R^{n+1}$, we have the Gauss map also given by
\begin{equation}
\gamma:M\longrightarrow G^+_{k,n}
\end{equation}
and in the case of hypersurfaces, the target becomes $S^{k+1}$.

In view of the above proposition, to obtain such subset $V\subset G^{+}_{p,n}$, one tries to find strictly convex functions $f:V\longrightarrow \R$ and conditions on the domain $M^k$ so that its Gauss image satisfies $\gamma(M)\subseteq V$. This strategy has been used vastly, as one can see in Hildebrandt-Jost-Widman, Jost-Xin, and others \cite{hildebrandt80}, \cite{jost99}, \cite{jost2016}.

As we have discovered in \cite{assimos2018}, instead of constructing a strictly convex function, we can exploit the geometry of such subsets admitting no image of harmonic maps. An explanation on how this is done comes on the next section.

\subsection{Non-existence of harmonic maps (into subsets of spheres)}

When studying Bernstein problems with Gauss map restrictions or slope bounds, more important than the existence of harmonic maps into a given target is the non-existence of those maps into some given subset of the target.

As discussed in the previous section, we want to find subsets of an oriented Grassmannian $G^+_{p,n}$, with $p$ and $n$ suited to the specific problem, that admit no image of harmonic maps, with the exception of the constant ones. We follow our previous work \cite{assimos2018} to exploit the geometry of such subsets.  

We start recalling Sampson's maximum principle \cite{sampson78}.

\begin{thm}\label{SMP}
 Let $(M,g)$, $(N,h)$ be, respectively, a compact and a complete Riemannian manifolds. Consider a non-constant harmonic map $\phi : M \longrightarrow N$ and suppose that at a given point $x\in M$, there exists a hypersurface $S\subset N$ with definite second fundamental form at the point $y\overset{def}{=} \phi(x)$. Then no neighborhood of $x\in M$ can be mapped entirely to the concave side of $S$.
\end{thm}

\begin{remark}
For a geodesic ball $B(p,r)$ in a given complete Riemannian  manifold $(N,h)$, with $p\in N$ arbitrary and $0< r< c_N$, where $c_N$ is defined to be the convexity radius of $N$, we have that $\partial B(p,r)$ is a hypersurface of $N$ with definite second fundamental form at each point $q\in\partial B(p,r)$.         
\end{remark}

The following theorem \cite{assimos2018}, is an application of Sampson's maximum principle that we intend to use extensively throughout the proofs on this paper.

\begin{thm}\label{corollary of SMP}
        Let $\Gamma: [a,b]\longrightarrow N$ be a smooth embedded curve on a complete Riemannian manifold $(N,h)$. Let $r:[a,b]\longrightarrow \R_+$ be a smooth function and consider a region on $N$ defined as follows.
        \begin{equation}\mathcal{R}\overset{\text{def}}{=}\bigcup_{t\in[a,b]}B(\Gamma(t),r(t)),\end{equation}
        where $B(\cdot,\cdot)$ denotes a geodesic ball and $r(t) < c_N$ for every $t\in [a,b]$. Moreover, suppose that for each $t_{0}\in (a,b)$, the set $\mathcal{R} \backslash B(\Gamma(t_{0}),r(t_{0}))$ is the union of two disjoint connected sets, namely the connected components containing $\Gamma(a)$ and $\Gamma(b)$, respectively. Then there exists no compact manifold $(M,g)$ and non-constant harmonic map $\phi:(M,g)\longrightarrow (N,h)$ such that  $\phi(M)\subset\mathcal{R}$.
\end{thm}

We refer to our previous work \cite{assimos2018} for a detailed proof of this statement. The following are consequences and interpretations of this result.

\begin{cor}\label{corollary1}
        Let $\mathcal{R}$ be a region on a complete Riemannian manifold $(N,h)$ that admits a sweepout $\{S_t\}_{t\in[a,b]}$ by convex hypersurfaces $S_t$ with the following properties:
        \begin{enumerate}[label=(\roman*)]
                \item Fix $t\in [a,b]$ arbitrary. For every $\epsilon > 0$, the leaf $S_{t-\epsilon}$ lies in the concave side of $S_t$;
                
                \item For each $t\in [a,b]$, the set $\mathcal{R}\setminus S_t$ is the union of two disjoint connected sets.
        \end{enumerate}
        Then, there is no compact Riemannian manifold $(M,g)$ and non-constant harmonic map $\phi: M \longrightarrow N$, with $\phi(M)\subset \mathcal{R}$. 
\end{cor}

\noindent  As an illustration, let us apply Theorem~\ref{corollary of SMP} and Corollary~\ref{corollary1} to spheres (see \cite{assimos2018} and \cite{jost2012}).  Using different techniques, it was shown that in $S^2\setminus(S^1\bigslant{}{\sim})_{\epsilon>0}$, where $\sim$ is the antipodal identification, there are no image of non-constant harmonic maps defined on a closed manifold. We actually show a much more general result, implicit in the following examples.  .

\begin{prop}\label{S2 without half equator}
        Let $p$ and $\mathscr{A}(p)$ be two points in $(S^2,\mathring{g})$, where $\mathscr{A}: S^2 \longrightarrow S^2$  denotes the antipodal map. Let $c: [0,1]\longrightarrow S^2$ be a connected curve such that $c(0)=p$, and $c(1)=\mathscr{A}(p)$.\footnote{it is not necessary, but to help the intuition one can suppose that $c([0,1])$ is contained in a hemisphere with respect to $p$ and $\mathscr{A}(p)$} For any $\epsilon>0$, there are no non-constant harmonic maps defined on closed manifolds with image in $S^2\setminus\left(c([0,1])\right)_{\epsilon}$, where $\left(c([0,1])\right)_{\epsilon}$ is defined as follows.
        \begin{equation*}
        \left(c([0,1])\right)_{\epsilon}\overset{\text{def}}{=}\{x\in S^2\,|\,\, d_g\left(x,c\left([0,1]\right)\right)<\epsilon\}.
        \end{equation*}
\end{prop}
        
\begin{proof}
        Fix $\epsilon > 0$ and define
        \begin{equation*}
        \mathcal{R}\overset{\text{def}}{=}\displaystyle\bigcup_{t\in S^1}\left( \partial B\left(\Gamma(t),\frac{\pi}{2}-\frac{\epsilon}{2}\right)\bigslant{}{\sim}\right),
        \end{equation*}
        where $\Gamma(t)$ is the great circle such that $\langle\Gamma(t),p\rangle=0$, for every $t\in S^1$.
        Although $\mathcal{R}$ is by definition a sweepout by convex hypersurfaces, $\mathcal{R}$ obviously does not satisfy condition $(ii)$ in Corollary~\ref{corollary1}. But it is clear that $\mathcal{R}\setminus\left(c([0,1])\right)_{\epsilon} = S^2\setminus\left(c([0,1])\right)_{\epsilon}$ does satisfy $(ii)$. Therefore, there are no non-constant harmonic maps with image in $S^2\setminus\left(c([0,1])\right)_{\epsilon}$, see Figure~\ref{Fig2}. In particular, there are no closed geodesics in that region as well.
\end{proof}

\begin{figure} 
        \begin{picture}(100,130)
        \put(50,-110){\includegraphics[width=0.8\linewidth]{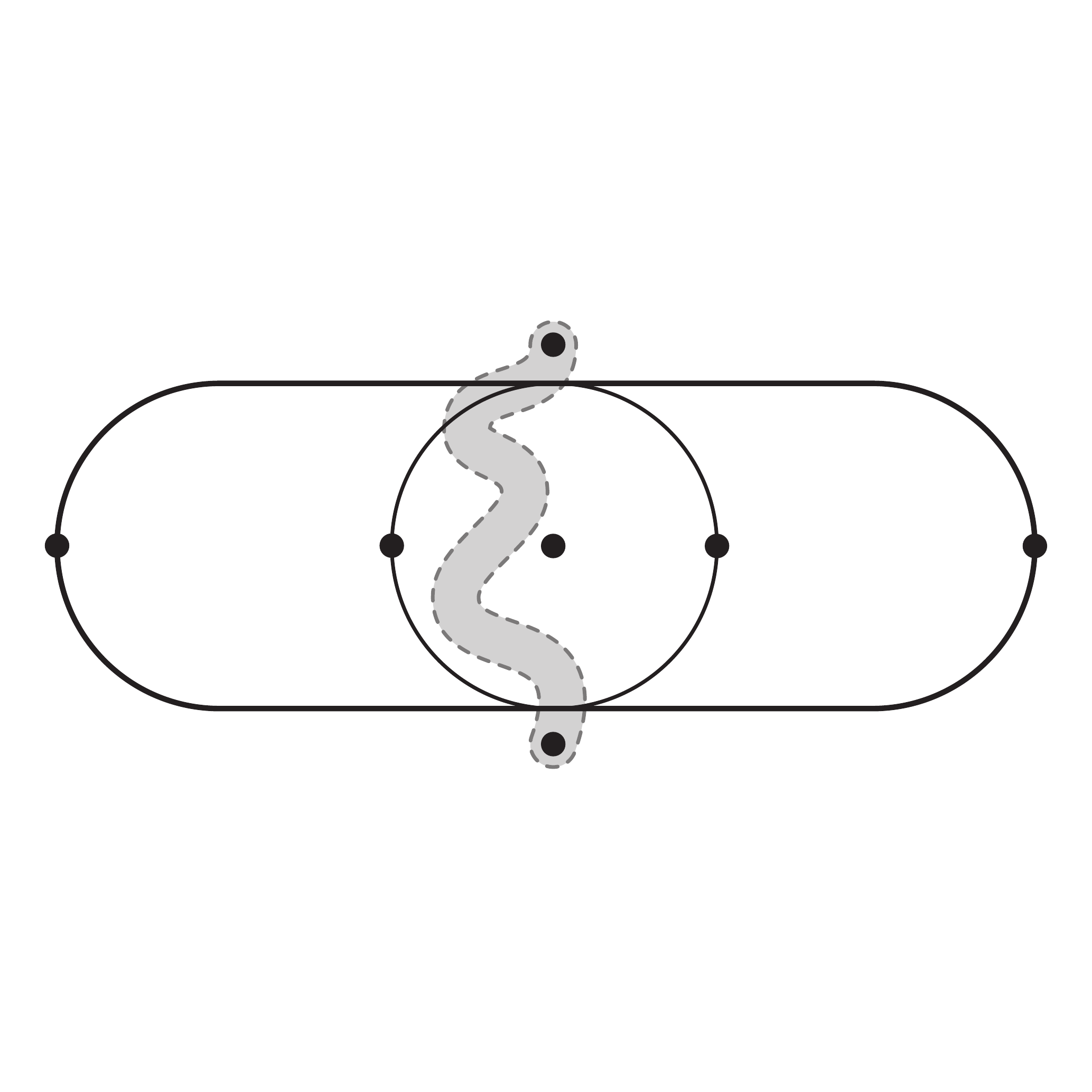}}
        \put(240,65){$e_3$}
        \put(397,65){$-e_3$}
        \put(43,65){$-e_3$}
        \put(240,133){$e_1$}
        \put(240,0){$-e_1$}
        \put(163,65){$e_2$}
        \put(292,65){$-e_2$}
        \end{picture}
        \caption{A projection of the convex sweepout of $S^2\setminus(\gamma([0,1]))_{\epsilon>0}$}
        \label{Fig2}
\end{figure}

The main idea of either the  proof of Theorem~\ref{corollary of SMP} or the above proposition is that $\partial \mathcal{R}$ is a barrier to the existence of non-constant harmonic maps. Since $\mathcal{R}$ is defined as the union of convex balls, we use Sampson's maximum principle to push the image of a possible non-constant harmonic map to this barrier. That is exactly what we do in the above example.

The classical method of looking for a strictly convex function $f:\mathcal{R}\longrightarrow \R$ is not very flexible. Once one changes the boundary of $\mathcal{R}$ slightly, one can no longer guarantee that there exists a  strictly convex function $\tilde f:\mathcal{\tilde R} \longrightarrow \R$.

Before the main theorem of this section regarding existence of harmonic maps into spheres, let us present an example. This could be regarded as a lemma that will help our intuition for the proof of the main result.

\begin{eg}\label{Sk without half equator}
        The argument in Proposition~\ref{S2 without half equator} can be adapted to the case of $S^{k+1}\setminus\left(S^{k}(x_1)\bigslant{}{\sim}\right)_{\epsilon>0}$. Here $x_1\in S^{k+1}$, $S^k(x_1)\overset{\text{def}}{=}\{p\in S^{k+1}\,|\,\, \langle p,x_1\rangle=0\}$ is an equatorial $S^k$, and $\{x_1,x_2,...,x_{k+2}\}$ is an orthonormal basis for $\R^{k+2}$. Using the region
    \begin{equation}\label{definition of R^j}
    \mathcal{R}^{k+2}\overset{\text{def}}{=}\displaystyle\bigcup_{t\in S^1}\left( \partial B\left(\Gamma^{k+2}(t),\frac{\pi}{2}-\frac{\epsilon}{2}\right)\bigslant{}{\sim}\right)
    \end{equation}
    determined by the great circle $\Gamma^{k+2}(t)\overset{\text{def}}{=}\cos(t)x_1+\sin(t)x_{k+2}$, we have that $\mathcal{R}^{k+2}=S^{k+1}\setminus(S^{k-1}(x_1,x_{k+2}))_{\epsilon>0}$, where $S^{k-1}(x_1,x_{k+2})$ denotes a totally geodesic codimension 2 sphere orthogonal to $x_1$ and $x_{k+2}$. Therefore, $(S^{k}(x_1)\bigslant{}{\sim})$ is clearly a barrier for the existence of non-constant harmonic maps.
\end{eg}

We can also find more flexible barriers like $\gamma([0,1])_{\epsilon>0}$ in the case of higher dimensions, although some extra care on the choice of the barrier is needed. Namely, we must understand the antipodal points $p,\varphi(p)$ in Example~\ref{S2 without half equator} as a $S^0$ sphere, that is, a codimension 2 totally geodesic subset. More precisely, we have the following.

\begin{thm}[Maximal sets in $S^k$]\label{Maximal sets in $S^k$}
	Let $C$ be an open connected subset of $S^{k+1}$ such that there exists an equatorial embedding $\dot\iota:S^{k-1}\longrightarrow S^{k+1}$, with $\dot\iota(S^{k-1})\subset C$. In addition, suppose that $S^{k-1}$ is homotopic (in $C$) to a point $x\in C$. Then there is no closed Riemannian manifold $M$ and non-constant harmonic map $\phi:M\longrightarrow S^{k+1}$ such that $\phi(M)\subset S^{k+1}\setminus C$.
\end{thm}

\begin{rmk}
	If $k=1$, we say that $S^0=\{1,-1\}$ is embedded in $S^2$ as an equator if $ \dot\iota(1)=\mathscr{A}(\dot\iota(-1))$. With this definition, one can take $C$ as some $\epsilon$-neighborhood of a curve $c:[a,b]\longrightarrow S^2$ connecting the two antipodal points in $S^0\hookrightarrow S^2$, getting Proposition~\ref{S2 without half equator} as a direct corollary of the above theorem.
\end{rmk}

\begin{proof}[Proof of Theorem~\ref{Maximal sets in $S^k$}.]
	
	Let $X$ be a subset of $S^{k+1}$ containing $k-1$ antipodal points; we write $X = \{x_1,-x_1,...,x_{k-1},-x_{k-1}\}$. Moreover, suppose that whenever $i\neq j$, we have that $\langle x_i,x_j\rangle=0$.
	
	Since $S^{k-1}$ is homotopic in $C$ to a point, without loss of generality, we can pick $x_0\in C\setminus S^{k-1}$ such that $\langle x_0,y\rangle=0$, for every $y\in S^{k-1}$.
	
	For a point $z_0\in S^{k+1}$, let us define the set
	\begin{equation*}
	S^k[z_0]\overset{\text{def}}{=}\{a\in S^{k+1}\,|\,\, \langle z_0,a\rangle=0\},
	\end{equation*}
	and analogously for a point $z_1\in S^k[z_0]$, define
	\begin{equation*}
	S^{k-1}[z_0,z_1]\overset{\text{def}}{=}\{a\in S^{k+1}\,|\,\, \langle z_0,a\rangle=\langle z_1,a\rangle=0\}.
	\end{equation*} 
	
	\noindent Recursively, we may define the sets 
	\begin{equation*}
	S^{k-\alpha}[z_0,z_1,...,z_\alpha]\overset{\text{def}}{=}\{a\in S^{k+1}\,|\,\, \langle z_0,a\rangle=\langle z_1,a\rangle=...=\langle z_\alpha,a\rangle=0\}, 
	\end{equation*}
	for a point $z_\alpha\in S^{k-(\alpha-1)}[z_0,z_1,...,z_{\alpha-1}]$, and $1\leq \alpha\leq k+1$.
	
	Using the above definition, we can use the $(k-1)$ non-antipodal points in $X$ and set $x_0=x$, where $x$ was chosen above. Moreover, we can still pick a point $\bar{x}\in S^{k-(k-1)}[x,x_1,...,x_{k-1}]$, a unit tangent vector $V\in T_{\bar{x}}S^{k+1}$, and a closed geodesic
	\begin{equation*}
	\Gamma_{V,\bar{x}}:[0,1]\longrightarrow S^{k+1} 
	\end{equation*}
	such that $\Gamma_{V,\bar{x}}(0)=\Gamma_{V,\bar{x}}(1)=\bar{x}$, $\Gamma_{V,\bar{x}}(\pi/2)=x$, and $\dot\Gamma_{V,\bar{x}}(0)=V$. As well, for every $t\in[0,1]$, we have that $\Gamma_{V,\bar{x}}(t)\in S^{k-(k-2)}[x_1,...,x_{k-1}]$.
	
	Like in Example~\ref{Sk without half equator}, we shall now consider the following region in $S^{k+1}$
	\begin{equation}\label{eq for R in proof of main theorem}
	\mathcal{R}\overset{\text{def}}{=}\displaystyle\bigcup_{t\in [0,1]}\left( \bigslant{\partial B\left(\Gamma_{V,\bar{x}}(t),\frac{\pi}{2}-\frac{\epsilon}{2}\right)}{\sim}\right)
	\end{equation}
	where $\epsilon>0$ is given such that $d(S^{k-1},C)<\epsilon$.
	
	By definition, $\mathcal{R}$ given by Equation~\ref{eq for R in proof of main theorem} is the sweepout of a 1-parameter family of convex hypersurfaces satisfying condition $(i)$ in Corollary~\ref{corollary1}. Moreover, since $C$ is an open set containing a codimension two equatorial sphere, $S^{k+1}\setminus \partial C$ is the disconnected union of two open sets, each of them connected.
	
	Now, one can easily see that
	\begin{equation*}
	\mathcal{R} = S^{k+1}\setminus \left(S^{k-1}\right)_{\epsilon}\supset S^{k+1}\setminus C,
	\end{equation*} 
	and since $\epsilon$ is chosen so that $\partial C \cap (S^{k-1}_{\epsilon})=\emptyset$, it follows that $\mathcal{R}\setminus\partial C$ is a region satisfying conditions $(i)$ and $(ii)$ of Corollary~\ref{corollary1}, proving there exist no non-constant harmonic map defined on compact Riemannian manifolds with image in the complement of $C$.
\end{proof}

\begin{rmk}[A quasi-counterexample]
        The Clifford torus in $S^3$ decomposes the sphere into two regions that admit mean convex sweepouts, for instance the equidistant one. By Sampson's maximum principle, there cannot exist other minimal hypersurfaces in one of these two regions. In fact, it is known more generally that every two minimal hypersurfaces intersect on manifolds with positive Ricci curvature. On the other hand, these two regions have several closed geodesics; or in other words, many codimension 2 totally geodesic spheres. Those geodesic $S^1$ are not homotopic to a point inside the region containing them.
\end{rmk}

\section{Solomon's Bernstein theorem}

\subsection{Codimension 1}

B. Solomon \cite{solomon84} has proven the following theorem.

\begin{thm}[Solomon]\label{solomon}
        Let $M^{k}\subset S^{k+1}$ be a smooth, compact minimal hypersurface. If $H^1(M)=0$ and the Gauss image of $M$is contained on a thickening of a great codimension two sphere $S^{k-1}$, $\gamma(M)\subset S^{k+1}\setminus \left(S^{k-1}\right)_{\epsilon}$, then $M$ is a totally geodesic hypersurface in $S^{k+1}$.
\end{thm}

We present a simple proof for this result as an application of Theorem~\ref{corollary of SMP}.

\begin{proof}
        Since $M$ is minimal, its Gauss map $\gamma:M\longrightarrow S^{k+1}$ is a harmonic map. Let $\tilde{M}$ be the universal cover of $M$, and denote by $\widetilde{(S^{k+1}\setminus S^{k-1})}$ the universal cover of $S^{k+1}\setminus S^{k-1}$. Let $\psi_1:\tilde{M}\longrightarrow M$ and $\psi_2:\widetilde{(S^{k+1}\setminus S^{k-1})}\longrightarrow (S^{k+1}\setminus S^{k-1})$ be the isometries given by the respective covering maps. Since $S^{k-1}$ has codimension 2 in $S^{k+1}$, its complement is not simply connected; in fact, $\widetilde{(S^{k+1}\setminus S^{k-1})}$ is an infinite strip in $R^{k+1}$ with metric $\psi_2^{*}(\mathring{\gamma})$, that is, the pull-back of the round metric $\mathring{\gamma}$ by the covering map.

        Since $H^1(M)=0$, we have that $\Pi_1(\gamma\circ \psi_1)=0$, where $\gamma\circ \psi_1: \tilde{M}\longrightarrow S^{k+1}$. This implies that the Gauss map $\gamma$ lifts to the universal cover as a harmonic map $\tilde{\gamma}:\tilde{M}\longrightarrow \widetilde{(S^{k+1}\setminus S^{k-1})}$, and $\tilde{\gamma}(\tilde{M})$ is a compact subset of $\widetilde{(S^{k+1}\setminus S^{k-1})}$. But the latter admits a sweepout by convex hypersurfaces, given by the lifting of the region in $S^{k-1}\setminus (S^{k-1})_{\epsilon}$ given by equation~\eqref{definition of R^j}. The theorem now follows from Corollary~\ref{corollary1}.
\end{proof}

\subsection{Codimension 2}

\subsubsection{A fast visit to the Grassmannian world}

To study the case of codimension 2, we need to understand the differential geometry of Grassmannian manifolds. We follow D. Hoffman and R. Osserman \cite{hoffman80}, S. Kozlov\cite{kozlov97}, and the work of the second author with Y. Xin \cite{jost99}.

The oriented Grassmanian $G^+_{2,k+2}$ has a natural orientation induced from a complex structure that can be defined as follows. Given an oriented 2-plane $P$ in $\R^{k+2}$, let $v,w$ be orthonormal vectors such that $P=\text{span}\{\langle v,w\rangle\}$. Define
\begin{equation*}
z=v +iw,
\end{equation*}
and note that this complex vector assigns a point of $\C^{k+2}$ to $P$. If one rotates $v$ and $w$ in $P$ by an angle $\theta$, we assign the complex vector $e^{i\theta}$ to $P$. Therefore, each oriented 2-plane $P$ is assigned to a unique point in the complex projective space $\C P^{k+1}$. Since $v$ and $w$ are orthonormal, it follows that
\begin{equation}\label{eq of Q_k}
\sum_{j=1}^{k+2}z^2_j=0,
\end{equation} 
where $z_j=v_j +iw_j$ for every $j\in\{1,...,k+2\}$. The above equation defines a quadric $Q_k\subset \C P^{k+1}$.

\noindent Moreover, if
\begin{equation}
ds^2= 2\frac{\sum_{j<l}|z_jdz_l-z_ldz_j|^2}{[\sum_{j=1}^{k+2}|z_j|^2]^2}
\end{equation}
is the Fubini-Study metric on $\C P^{k+1}$, and $(Q_k,\text{d}s^2|_{Q_k})$ is the pullback to $Q_k$ of the Fubini-Study metric with respect to the canonical embedding of $Q_k$ in $\C P^{k+1}$, we obtain an isometry between $G^+_{2,k+2}$ endowed with the metric of a homogeneous space and $(Q_k,\text{d}s^2|_{Q_k})$.

\begin{thm}[Hoffman-Osserman]\label{hoffman-osserman}
        Let $H$ be the hyperplane in $\C P^{k+1}$ given by $H: z_1 -iz_2=0$, and 
        \begin{equation*}
        Q_k^{*}\overset{\text{def}}{=} Q_k\setminus H .
        \end{equation*}
        There exists a biholomorphic map $\varphi: Q_k^{*}\longrightarrow \C^k$ given by
        \begin{equation}
        (z_1,...,z_{k+2})=\frac{\left(z_1 - iz_2\right)}{2}\left(1-\sum_{j=1}^{k}\xi^2_j, i\left(1+\sum_{j=1}^{k}\xi^2_j\right),2\xi_1,...,2\xi_k\right)
        \end{equation}
        where
        \begin{equation}
        \xi_1=\frac{z_3}{z_1 - iz_2},...,\xi_k=\frac{z_{}k+2}{z_1 - iz_2}
        \end{equation}
\end{thm}

A proof can be found in D. Hoffman and R. Osserman \cite{hoffman80}.

Consider another hyperplane $H^{\prime}:z_1+iz_2=0$ in $\C P^{k+1}$. Obviously, $Q_k\cap(H\cup H^{\prime})=\{((z_1,...,z_{k+2}))\in Q_k\,|\,\, z^2_1+z^2_2=0\}$ has codimension 2, and therefore
\begin{equation*}
Q_k\setminus (H^{\prime}\cup H)=\C^k\setminus \varphi(H^{\prime})
\end{equation*}
is not simply connected.

For the case $k=2$, we have $Q_2=G^+_{2,4}$ and
\begin{equation}\label{Q_2 without two hyperplanes}
Q_2\setminus (H^{\prime}\cup H)=S^2\times S^2\setminus \left[\left(S^0\times S^2\right)\cup\left(S^2\times S^0\right)\right].
\end{equation}

To find convex sets in the general quadric $Q_k$, we need to introduce some notation and present results by S. Kozlov, and the second author with Y. Xin \cite{kozlov97}, \cite{jost99}.

Let $\psi: G^+_{p,n} \longrightarrow \Lambda_p(\R^n)$ denote the Pl\"ucker embedding of the oriented Grassmannian into the Grassmann algebra. We identify $G^+_{p,n} \simeq \psi(G^+_{p,n})$ without further comments, whenever it will not cause misunderstandings.

Remember that $G^+_{p,n} = \psi(G^+_{p,n}) = K_p \cap S^{\binom{n}{p}-1}\subset \R^{\binom{n}{p}}\simeq \Lambda_{p}(\R^n)$, where $K_p$ is the cone of simple $p$-vectors in $\Lambda_p(\R^n)$.

For a point $w \in G^+_{p,n}$, we take orthonormal basis $\{e_i\}_{i=1}^{p}$, $\{n_\alpha\}_{\alpha=1}^{q}$, for $w$ and $w^{\perp}$, respectively. We will denote by $\{\eta_{i\alpha}\}_{i=1,...,p}^{\alpha=1,...,q}$ the following system of vectors. 
\begin{equation}
\eta_{i\alpha} = e_1\wedge \cdot\cdot\cdot\wedge e_{i-1}\wedge n_\alpha \wedge e_{i+1}\wedge \cdot\cdot\cdot \wedge e_p.
\end{equation}

It is easy to check that such a system forms a basis for the tangent space $T_{w}G^{+}_{p,n}$. We denote an element of the tangent bundle $TG^{+}_{p,n}$ by $(w,X)$. In terms of the basis vectors, one has the following

\begin{thm}[Kozlov \cite{kozlov97}]
        Let $(w,X)\in TG^{+}_{p,n}$, with $X\neq 0$. Then there exists an integer $r\in\{1,...,\min\{p,q\}\}$, an orthonormal basis $\{e_{i}\}_{i=1}^{p}$ for $w\in G^+_{p,n}$, and a system $\{m_{i}\}_{i=1}^{r}$ of non-zero pairwise orthogonal vectors in $span\langle\{n_\alpha\}\rangle$, such that
        \begin{align}
        w =& e_{1}\wedge...\wedge e_{p}, \label{local expression of w}\\
        X =& \left(m_{1}\wedge e_{2}\wedge...\wedge e_{r}+...+e_{1}\wedge...\wedge e_{r-1}\wedge m_{r}\right)\wedge\left(e_{r+1}\wedge...\wedge e_{p}\right).\label{local expression of X}
        \end{align}
\end{thm}

Writing $m_\alpha = \lambda^\alpha n_\alpha$, we have that $\lambda^\alpha=|m_\alpha|$. With this notation, we may compute
\begin{equation}\label{equation for X}
X=\left(\lambda^{1}n_{1}\wedge e_{2}\wedge...\wedge e_{r} +...+ \lambda^{r}e_{1}\wedge ... \wedge e_{r-1}\wedge n_{r}\right)\wedge X_{0},
\end{equation}
obtaining $\left|X\right|^2=\left(\sum_{\alpha=1}^{r}\left|\lambda^{\alpha}\right|^{2}\right)$.

For a brief geometrical description of geodesics in $G^+_{p,n}$, see S. Kozlov \cite{kozlov97} and for a  more detailed treatment see  our previous work \cite{assimos2018}. The geodesics are obtained simply by rotation of tangent basis vectors into normal ones; we consider 
\begin{align}
& e_i(s)\overset{\text{def}}{=} e_i\cos(s) + n_i\sin(s)\\
& n_i(s)\overset{\text{def}}{=} -e_i\sin(s) + n_i\cos(s)
\end{align}

Then a geodesic passing through $w=e_1\wedge\cdot\cdot\cdot\wedge e_p \in G^+_{p,n}$ with tangent vector $X=\left(\lambda^1n_1\wedge e_2\wedge\cdot\cdot\cdot\wedge e_r + ... + \lambda^r e_1\wedge\cdot\cdot\cdot \wedge e_{r-1}\wedge n_r\right)\wedge\left(e_{r+1}\wedge\cdot\cdot\cdot\wedge e_p\right)\in T_wG^+_{p,n}$, can be written as
\begin{equation}
w_{X}(t)\overset{\text{def}}{=} \left(e_1(\lambda^1 t)\wedge \cdot\cdot\cdot \wedge e_r(\lambda^r t)\right)\wedge\left(e_{r+1}\wedge\cdot\cdot\cdot\wedge e_p\right)
\end{equation}

Now, let us assume that $X\in T_w G^+_{p,n}$ is always a unit speed vector, i.e., 
\begin{equation*}
\left|X\right|^2=\left(\sum_{\alpha=1}^{r}\left|\lambda^{\alpha}\right|^{2}\right)=1.
\end{equation*}

Denoting by $\lambda_{\alpha^{\prime}}\overset{\text{def}}{=} \underset{\alpha}{\max}\{\lambda^{\alpha}\}$ and $\lambda_{\beta^{\prime}}\overset{\text{def}}{=} \underset{\alpha\neq \alpha}{\max}\{\lambda^{\alpha}\}$, we define the following number.
\begin{equation}
t_{X}\overset{\text{def}}{=}\frac{\pi}{2\left(\left|\lambda_{\alpha^{\prime}}\right| + \left|\lambda_{\beta^{\prime}}\right|\right)},
\end{equation}

\begin{thm}[Jost-Xin]\label{Jost-Xin}
        Let $(w,X)\in T G^{+}_{p,n}$, where $X$ is a unit tangent vector as in Equation~\eqref{equation for X}. The set defined by
        $$
        B_{G}(w)\overset{\text{def}}{=}\{w_X(t)\in G^{+}_{p,n}; 0\leq t\leq t_{X}\}.
        $$
        is a convex set (in the geodesic sense) and contains the largest geodesic ball centered at $w$.
\end{thm}

\subsubsection{The main theorem}\label{main}

We start giving a proof for a much simpler case as a geometric motivation. Let $M^2\subset S^4$ be a codimension 2 compact minimal immersed submanifold such that $H^1(M)=0$. If the Gauss map $\gamma:M\longrightarrow Q_2$ omits two hyperplanes as in Equation~\ref{Q_2 without two hyperplanes}, then $\gamma$ is constant and $M$ a totally geodesic 2-sphere in $S^4$.

To prove this, define
\begin{equation*}
\mathcal{R}\overset{\text{def}}{=}\displaystyle\bigcup_{t,s\in S^1} \partial B\left(\Gamma(t),\frac{\pi}{2}-\frac{\epsilon}{2}\right)\times \partial B\left(\Gamma(s),\frac{\pi}{2}-\frac{\epsilon}{2}\right),
\end{equation*}
where $\partial B\left(\Gamma(t),\frac{\pi}{2}-\frac{\epsilon}{2}\right)$ is given in Example~\ref{S2 without half equator} and the claim follows from Theorem~\ref{corollary of SMP}

Since we assume $H^1(M)=0$, we are basically assuming that $M$ is topologically $S^2$. But the harmonic Gauss map $\gamma:S^2\longrightarrow S^2\times S^2$ must omit a large set in $S^2\times S^2$, therefore the question of when a minimal immersion of $S^2$ into $S^4$ is totally geodesic is reduced to the question of finding two antipodal points in the image of the Gauss map, one in each of the $S^2$-components of the Grassmannian $G^+_{2,4}$.

We shall now state the main theorem. For that, we set $n=k+2$ in the above, and define
\begin{equation}\label{the main region in the grassmannian}
\mathcal{R}\overset{\text{def}}{=}\displaystyle\big(\displaystyle\bigcup_{t\in S^1}\partial B_{G}(w_{X_{1}}(t))\big).
\end{equation}
where $X_{1}=n_{1}\wedge e_{2}\wedge e_{3}\wedge...\wedge e_{p}$.

This region gives a sweepout by convex hypersurfaces of a set that contains none of the points $w_{X_{2}}(\pm t_{X_2})$, where $X_2$ is the unit tangent vector at $w$ given by the rotation of two vectors $e_i$ and $e_j$ to normal directions $n_\alpha$, $n_\beta$ (for example, $X_{2}=\left(\frac{\sqrt{2}}{2}n_{1}\wedge e_{2} + \frac{\sqrt{2}}{2}e_{1}\wedge n_{2}\right)\wedge e_{3}\wedge\cdot\cdot\cdot\wedge e_{p}$). Now, if we take the biholomorphism $\varphi$ of Theorem~\ref{hoffman-osserman}, then for an appropriate basis and two different vectors $X_{2}$ and $\tilde{X}_{2}$, that meaning two different rotations of basis vectors, $\varphi^{-1}(w_{X_{2}})\overset{\text{def}}{=} H$ and $\varphi^{-1}(w_{X_{2}})\overset{\text{def}}{=}H^{\prime}$. This construction gives us the following theorem, which can be seen as a codimension 2 version of Solomon's theorem.

\begin{thm}\label{mainthm}
        Let $M^k$ be a codimension 2 compact minimal submanifold of $S^{k+2}$ with $H^1(M)=0$. Suppose that its Gauss image is contained in the region $\mathcal{R}$ given by \eqref{the main region in the grassmannian}. Then $g$ is constant and $M$ a totally geodesic submanifold of $S^{k+2}$.   
\end{thm}

\begin{proof}
        By the above argument, $\mathcal{R}$ is contained in a region that is not simply connected in $Q_k$. As in the proof of Theorem~\ref{solomon}, the Gauss map lifts to a map $\tilde{g}: \tilde{M}\longrightarrow \widetilde{Q_k\setminus (H^{\prime}\cup H)}$, and the image $\tilde{g}(\tilde{M})$ is compact in $\tilde{\mathcal{R}}\subset\widetilde{Q_k\setminus (H^{\prime}\cup H)}$. Since we can lift the convex sweepout of $\mathcal{R}$ to $\tilde{\mathcal{R}}$, we conclude that $\tilde(g)$ is constant. Therefore $g$ is constant and $M$ is a totally geodesic submanifold of $S^{k+2}$.
\end{proof}

\begin{rmk}
	Note that in the codimension one case, $n=k+1$, the set $B_G(w)$ is a geodesic ball of radius $\frac{\pi}{2}$ and 
	$$
	\displaystyle\bigcup_{t\in S^1}\partial B_{G}\left(w_{X_{1}}(t)\right)
	$$
	is exaclty $S^{k+1}\setminus S^{k-1}[w]$, where $S^{k-1}[w]=\{p\in S^{k+1}\,|\,\, \left\langle p,w\right\rangle=\left\langle p,w_{X_1}\left(\pi/2\right)\right\rangle=0\}$.
\end{rmk}

\newpage 
\bibliographystyle{alpha}
\bibliography{bernstein}
\addcontentsline{toc}{section}{\bibname}

\end{document}